\documentclass[twoside,11pt]{amsart} \usepackage{epsfig}
\usepackage{latexsym} 
\usepackage{amsmath} \usepackage{amssymb}



\newtheorem{theorem}{Theorem}[section]
\newtheorem{lemma}[theorem]{Lemma}

\newdimen\epsfxsize
\newdimen\epsfysize

\newcommand {\gap}     {\makebox[0.075 in]{}}

\newcommand {\fto}     {\longrightarrow}

\newcommand {\ord}[1]  {{#1}^{\rm th}}

\newcommand{\primeprod}[1] {{#1}^\#}

\newcommand{\N}[2]  {N_{{#2}}({#1})}
\newcommand{\Rat}[2]  {w_{{#2},{#1}}}  % QHERE -- NOTATION!! 

\newcommand{\pgap}   {{\mathcal G}}

\newcommand{\Bi}[2]{\left( \begin{array}{c}{#1} \\ {#2} \end{array}\right)}
\newcommand{\lil}{\scriptstyle}

\begin{document}

\title{On small gaps among primes }
% \shorttitle{primorial gaps}
%  XXXQHERE - estimating g=30.

\date{8 Jan 2014- version 1.1}

%\institution{U of Washington}
\author{Fred B. Holt and Helgi Rudd}
\address{fbholt@uw.edu ;  4311-11th Ave NE \#500, Seattle, WA 98105;
48B York Place, Prahran, Australia 3181}

\begin{abstract}
A few years ago we identified a recursion that works directly with the gaps among
the generators in each stage of Eratosthenes sieve.  This recursion provides explicit
enumerations of sequences of gaps among the generators, which are known as 
constellations.

As the recursion proceeds, adjacent gaps within longer constellations are added together
to produce shorter constellations of the same sum.  These additions or closures correspond to
removing composite numbers that are divisible by the prime for that stage of Eratosthenes sieve.
Although we don't know where in the cycle of gaps a closure will occur, we can enumerate 
exactly how many copies of various constellations will survive each stage.

In this paper, we study these systems of constellations of a fixed sum.  Viewing them
as discrete dynamic systems, we are able to characterize the populations of constellations
for sums including the first few primorial numbers:  $2$, $6$, $30$.

Since the eigenvectors of the discrete dynamic system are independent of the prime -- that is,
independent of the stage of the sieve -- we can characterize the asymptotic behavior exactly.
In this way we can give exact ratios of the occurrences of the gap $2$ to
the occurrences of other small gaps for all stages of Eratosthenes sieve.
\end{abstract}

\maketitle

\section{Introduction}
We work with the prime numbers in ascending order, denoting the
$\ord{k}$ prime by $p_k$.  Accompanying the sequence of primes
is the sequence of gaps between consecutive primes.
We denote the gap between $p_k$ and $p_{k+1}$ by
$g_k=p_{k+1}-p_k.$
These sequences begin
$$
\begin{array}{rrrrrrc}
p_1=2, & p_2=3, & p_3=5, & p_4=7, & p_5=11, & p_6=13, & \ldots\\
g_1=1, & g_2=2, & g_3=2, & g_4=4, & g_5=2, & g_6=4, & \ldots
\end{array}
$$

A number $d$ is the {\em difference} between prime numbers if there are
two prime numbers, $p$ and $q$, such that $q-p=d$.  There are already
many interesting results and open questions about differences between
prime numbers; a seminal and inspirational work about differences
between primes is Hardy and Littlewood's 1923 paper \cite{HL}.

A number $g$ is a {\em gap} between prime numbers if it is the difference
between consecutive primes; that is, $p=p_i$ and $q=p_{i+1}$ and
$q-p=g$.
Differences of length $2$ or $4$ are also gaps; so open questions
like the Twin Prime Conjecture, that there are an infinite number
of gaps $g_k=2$, can be formulated as questions about differences
as well.

A {\em constellation among primes} \cite{Riesel} is a sequence of consecutive gaps
between prime numbers.  Let $s=c_1 c_2 \cdots c_k$ be a sequence of $k$
numbers.  Then $s$ is a constellation among primes if there exists a sequence of
$k+1$ consecutive prime numbers $p_i p_{i+1} \cdots p_{i+k}$ such
that for each $j=1,\ldots,k$, we have the gap $p_{i+j}-p_{i+j-1}=c_j$.  
Equivalently,
$s$ is a constellation if for some $i$ and all $j=1,\ldots,k$,
$c_j=g_{i+j}$.

We will write the constellations without marking a separation between
single-digit gaps.  For example, a constellation of $24$ denotes
a gap of $g_k=2$ followed immediately by a gap $g_{k+1}=4$.
For the small primes we will consider
explicitly, most of these gaps are single digits, and the separators introduce
a lot of visual clutter.  We use commas only to separate double-digit gaps in
the cycle.  For example, a constellation of $2,10,2$ denotes a gap of $2$
followed by a gap of $10$, followed by another gap of $2$.

In \cite{FBHgaps} we introduced a recursion that works directly on the gaps
among the generators in each stage of Eratosthenes sieve.
These are the generators of $Z \bmod \primeprod{p}$ in which $\primeprod{p}$
is the product of the prime numbers from $2$ through $p$, known as the {\em primorial} of $p$.
For a constellation $s$, this recursion enables us to enumerate exactly how
many copies of $s$ occur in the $\ord{k}$ stage of the sieve.  We denote this
number of copies of $s$ as $\N{p_k}{s}$.

For example, after the primes $2$, $3$, and $5$ and their multiples have been
removed, we have the cycle of gaps $\pgap(\primeprod{5}) = 64242462$.  This
cycle of $8$ gaps sums to $30$.  In this cycle, for the constellation $s=2$, we 
have $\N{5}{2}=3$.  For the constellation $s=242$, we have $\N{5}{242}=1$.
The cycle of gaps $\pgap(\primeprod{p})$ has $\phi(\primeprod{p})$ gaps that sum
to $\primeprod{p}$.

In \cite{HRest} we assumed that copies of a constellation were approximately uniformly
distributed within the cycle of gaps 
$\pgap( \primeprod{p})$, from which we could then estimate the numbers of 
these constellations that survive to occur as constellations
among prime numbers.  For a few select constellations we compared our
estimates to actual counts up through $10^{12}$.
For these constellations, our estimates in \cite{HRest} appear
to have the correct asymptotic behavior, but our estimates also seem to have a
systematic error correlated with the number of gaps in the constellation.

In this paper, we identify a discrete dynamic system that provides exact counts
of a gap and its driving terms, which are constellations that under successive 
closures produce the gap at later stages of the sieve.  These raw counts grow
superexponentially, and so to better understand their behavior we take the ratio
of a raw count to the number of gaps $g=2$ at each stage of the sieve.

For a gap $g$ that has driving terms of lengths $2 \le j \le J$, we form a vector
of initial values $\left. \bar{w}\right|_{p_{0}}$, whose $j^{\rm th}$ entry is the ratio
of the number of driving terms for $g$ of length $j$ in $\pgap(\primeprod{p_0})$
to the number of gaps $2$ in this cycle of gaps.
Recasting the discrete dynamic system to work with these ratios, we have
\begin{eqnarray*}
\left. \bar{w}\right|_{p_{k}} & = & \left. M_J \right|_{p_{k}} \cdot \left. \bar{w}\right|_{p_{k-1}} \\
 & = & M_J^k  \cdot \left. \bar{w}\right|_{p_{0}} 
\end{eqnarray*}
The matrix $M_J$ does not depend on the gap $g$.  It does depend on the prime $p_k$,
and we use the exponential notation $M_J^k$ to indicate the product of the $M$'s over
the indicated range of primes.

Although the matrix $M_J$ depends on the prime $p_k$, its eigenvectors do not.
We are therefore able to give a simple exact expression of the dynamic system that
reveals its asymptotic behavior.  We show that as $p_k \fto \infty$, the following ratios
describe the relative frequency of occurrence of gaps in Eratosthenes sieve:

\begin{center}
\begin{tabular}{rl}\hline
ratio $N_g / N_2$ : & gaps $g$ with this ratio \\ \hline
$1$ : & $2, \; 4, \; 8, \; 16, \; 32$  \\
$2$ : & $ 6, \; 12, \; 18, \; 24$ \\
$ 2.\bar{6}$ : & $ 30$
\end{tabular}
\end{center}

The ratios discussed in this paper give the exact values of the relative frequencies of various gaps and
constellations as compared to the number of gaps $2$ at each stage of Eratosthenes sieve.
As the sieving process continues, if the closures are at all fair, then these ratios should
also be good approximations to the relative occurrence of these gaps and constellations as gaps 
among primes.

\section{Recursion on Cycle of Gaps}
In the cycle of gaps, the first gap corresponds to the next prime.  In $\pgap(\primeprod{5})$
the first gap $g_1=6$, which is the gap between $1$ and the next prime, $7$.  The next
several gaps are actually gaps between prime numbers.  In the cycle of gaps
$\pgap(\primeprod{p_k})$, 
the gaps between $p_{k+1}$ and $p_{k+1}^2$ are in fact gaps between prime numbers.

There is a simple recursion which generates $\pgap(\primeprod{p_{k+1}})$ from
$\pgap(\primeprod{p_k})$.  This recursion and many of its properties are developed
in \cite{FBHgaps}.  We include only the concepts and results we need for 
developing the material in this paper.

The recursion on the cycle of gaps consists of three steps.
\begin{itemize}
\item[R1.] The next prime $p_{k+1} = g_1+1$, one more than the first gap;
\item[R2.] Concatenate $p_{k+1}$ copies of $\pgap(\primeprod{p_k})$;
\item[R3.] Add adjacent gaps as indicated by the elementwise product 
$p_{k+1}*\pgap(\primeprod{p_k})$:  let $i_1=1$ and add together $g_{i_1}+g_{i_1+1}$; then for 
$n=1,\ldots,\phi(N)$, add $g_{j}+g_{j+1}$ and let 
$i_{n+1}=j$ if the running sum of the concatenated gaps from $g_{i_n}$ to $g_j$ is
$p_{k+1}*g_{n}.$
\end{itemize}

\noindent{\bf Example: $\pgap(\primeprod{7})$.}
To illustrate this recursion,
we construct $\pgap(\primeprod{7})$ from $\pgap(\primeprod{5})=64242462$.

\begin{itemize}
\item[R1.] Identify the next prime, $p_{k+1}= g_1+1 = 7.$
\item[R2.] Concatenate seven copies of $\pgap(\primeprod{5})$:
$$\small 64242462 \; 64242462 \; 64242462 \; 64242462\; 64242462 \; 64242462 \;64242462$$
\item[R3.] Add together the gaps after the leading $6$ and 
thereafter after differences of 
$ 7*\pgap(\primeprod{5}) = 42, 28, 14, 28, 14, 28, 42, 14 $:
\begin{eqnarray*}
\pgap(\primeprod{7}) 
 &=&{\scriptstyle 
  6+\overbrace{\scriptstyle 424246264242}^{42}+
 \overbrace{\scriptstyle 4626424}^{28}+\overbrace{\scriptstyle 2462}^{14}+
 \overbrace{\scriptstyle 6424246}^{28}+\overbrace{\scriptstyle 2642}^{14}+
 \overbrace{\scriptstyle 4246264}^{28}+\overbrace{\scriptstyle 242462642424}^{42}+
 \overbrace{\scriptstyle 62 \; }^{14}} \\
 &=& {\scriptstyle 
 10, 2424626424 6 62642 6 46 8 42424 8 64 6 24626 6 4246264242, 10, 2}
\end{eqnarray*}
The final difference of $14$ wraps around the end of the cycle,
 from the addition preceding the final $6$ to the 
addition after the first $6$.
\end{itemize}

We summarize a few properties of the cycle of gaps $\pgap(\primeprod{p})$, as established in 
\cite{FBHgaps}.  The cycle of gaps ends in a $2$, and except for this final $2$, the cycle of gaps
is symmetric.  In constructing $\pgap(\primeprod{p_{k+1}})$,
each possible addition of adjacent gaps in $\pgap(\primeprod{p_k})$ occurs exactly once.

\subsection{Numbers of constellations}
The power of the recursion on the cycle of gaps is seen in the following theorem,
which enables us to calculate the number of occurrences of a constellation $s$ through
successive stages of Eratosthenes sieve.

\begin{theorem}\label{CountThm}
(from \cite{FBHgaps})
Let $s$ be a constellation of $j$ gaps in $\pgap(\primeprod{p_k})$, such that 
$j < p_{k+1}-1$ and  $\sigma(s) < 2p_{k+1}$.
Let $S$ be the set of all constellations $\bar{s}$ which would produce
$s$ upon one addition of adjacent gaps in $\bar{s}$.
Then the number $\N{p}{s}$ of occurrences of $s$ in $\pgap(\primeprod{p})$
 satisfies the recurrence
$$\N{p_{k+1}}{s} = (p_{k+1}-(j+1)) \cdot \N{p_k}{s}
 + \sum_{\bar{s} \in S} \N{p_k}{\bar{s}}$$
\end{theorem}

\begin{figure} 
\centering
\includegraphics[width=4in]{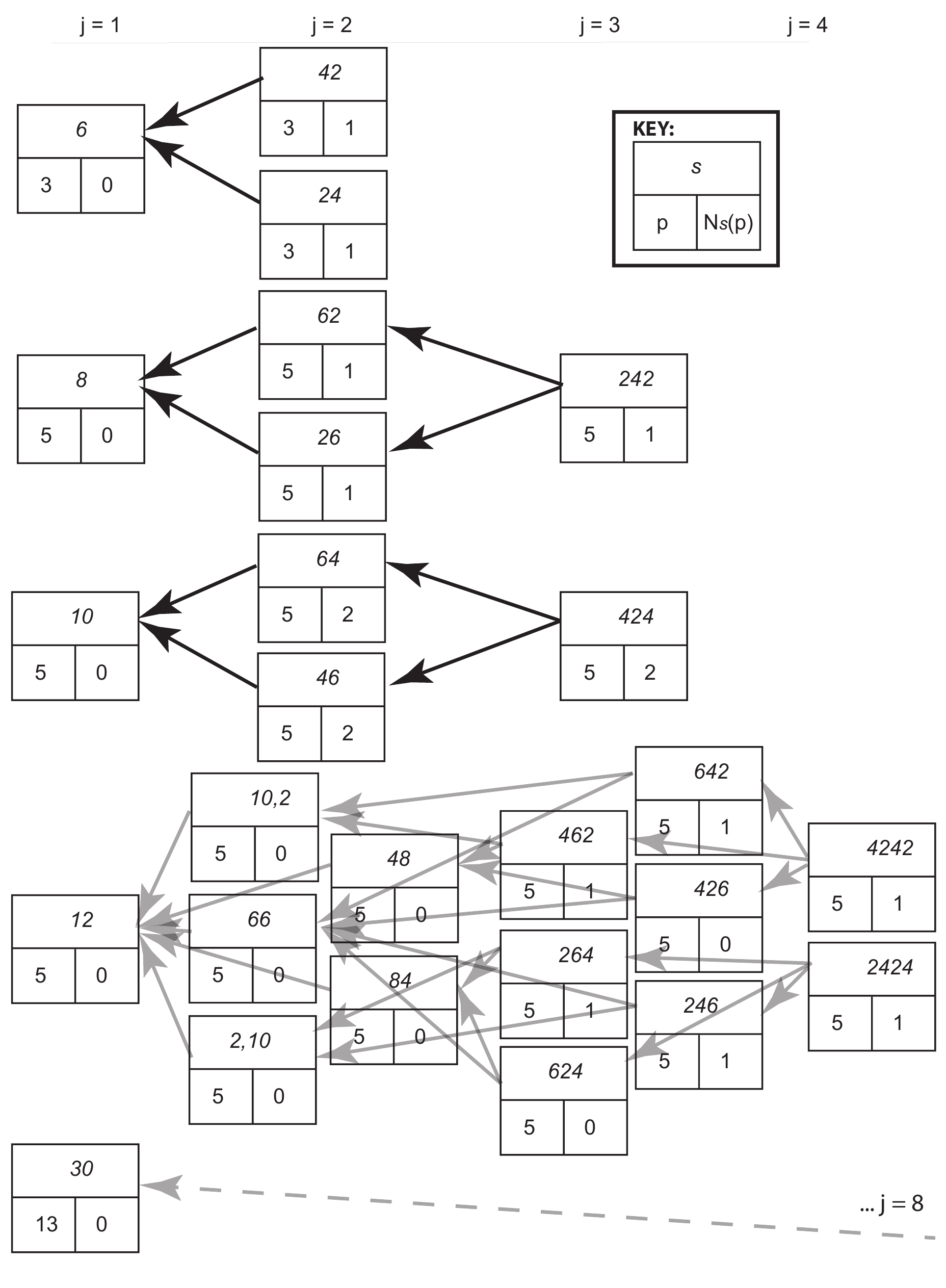}
\caption{\label{GrowthFig} This figure illustrates the initial conditions and driving terms for
calculating the numbers of copies of the gaps $6, \; 8, \; 10, \; 12$ in 
$\pgap(\primeprod{p})$.   The entries in this chart
 indicate the  constellation $s$, its length $j$;
the prime for which the constellation occurs in $\pgap(\primeprod{p})$ and which satisfies the
conditions of Theorem~\ref{CountThm}; and the number $N=\N{p}{s}$ of occurrences of the
constellation in $\pgap(\primeprod{p})$.  From these figures we can derive the recursive count
$\N{q}{s}$ for primes $q > p$.  For the gap $30$, the system of driving terms goes out to length
$j=8$.}
\end{figure}

\section{The dynamic system}

Figure~\ref{GrowthFig} illustrates the initial conditions for the gaps
$2$, $4$, $6$, $8$, $10$, and $12$, and their driving terms.  
Note that the initial conditions are not predicated on
when the constellations first appear but on the $\pgap(\primeprod{p})$ for which
the constellations satisfy the conditions of Theorem~\ref{CountThm}.

For larger gaps, these systems of driving terms become more unwieldy.  For a gap $g$, we don't need to
identify all of the individual constellations of length $j$ that sum to $g$.  All we need is a count of these
constellations.  So our diagram in Figure~\ref{GrowthFig} becomes simpler, 
as shown in Figure~\ref{SystemFig}.

\begin{figure}[t]
\centering
\includegraphics[width=5in]{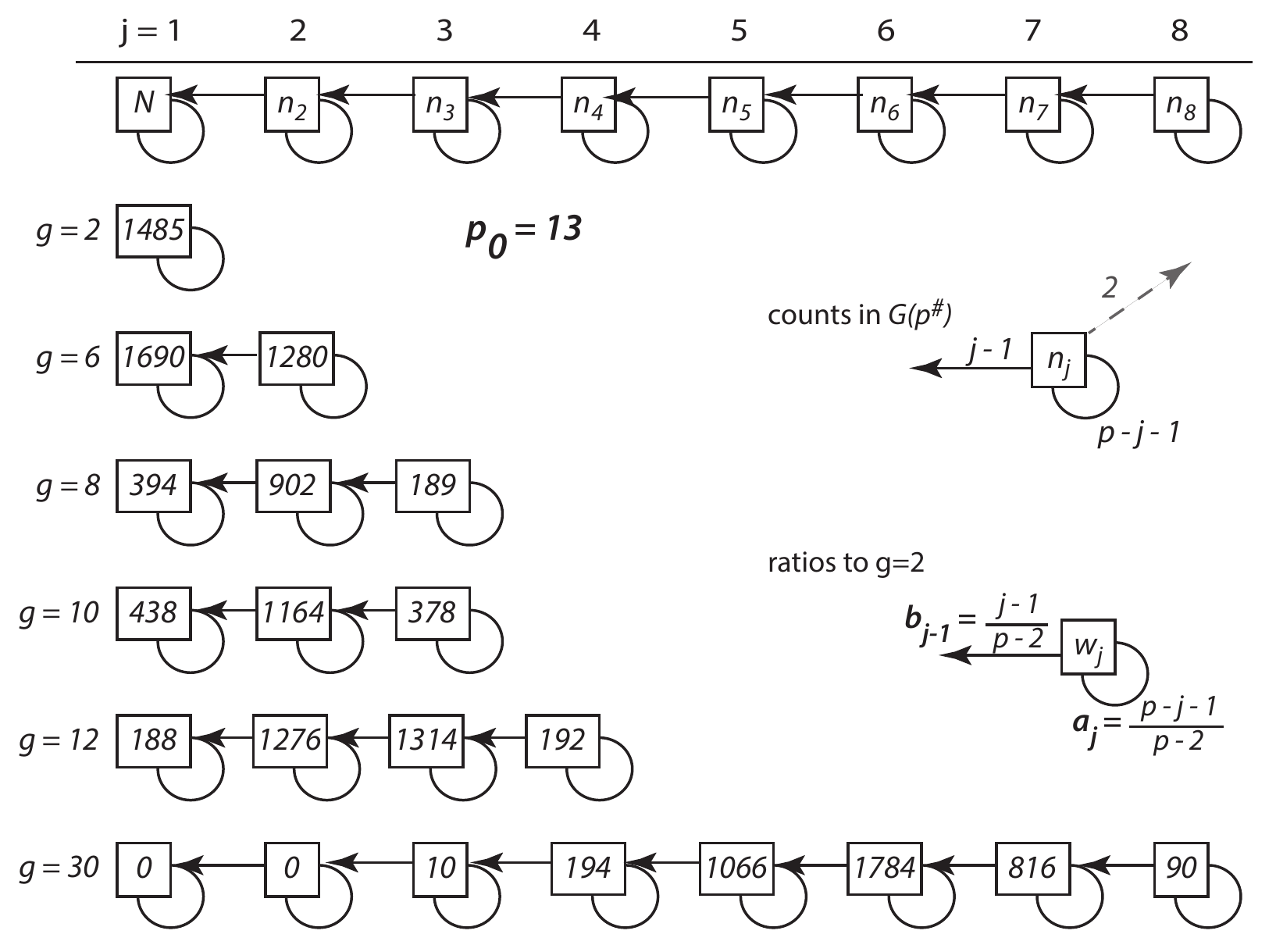}
\caption{\label{SystemFig} This figure illustrates the dynamic system of
Theorem~\ref{CountThm} through stages of the recursion for
$\pgap(\primeprod{p})$, using just the counts of gaps and their driving terms.
The action of the system at each stage of the recursion
is independent of the specific gap and its driving terms.  
Below the diagram for the system, we record the initial conditions
for a set of gaps at $p_0=13$.  From this information we can derive the recursive count
$\N{q}{s}$ for primes $q > p_0$.  Since the raw counts are superexponential, we take
the ratio of the count for each constellation to the simplest counts $\N{p}{2}=\N{p}{4}$. }
\end{figure}

Recall that $g=2$ has no driving terms, so
$$\N{p_{k}}{2} = (p_k - 2) \cdot \N{p_{k-1}}{2}.$$

Let $n_{s,j}(p)$ be the number of all constellations of length $j$ that either are copies of $s$ itself
(if $j$ equals the length of $s$) or are 
driving terms for $s$, in $\pgap(\primeprod{p})$.  As the recursion continues, these numbers
$n_{s,j}$ grow superexponentially by factors of $(p-j-1)$.  
To make the numbers and analysis manageable over many stages of the recursion, we
normalize by the number of $2$'s,  $\N{p}{2}=\N{p}{4}$.  We define
$$ \Rat{j}{s}(p) = n_{s,j}(p) / \N{p}{2}.$$

Anticipating our work with $g=30$ below, let us use $p_0=13$ for our initial conditions.
The prime $p=13$ is the first prime for which the conditions of Theorem~\ref{CountThm} 
are satisifed for $g=30$.  In $\pgap(\primeprod{13})$ we have the following
initial values.

\renewcommand\arraystretch{0.8}
\begin{center}
\begin{tabular}{|r|rrrrrrrrr|} \hline
gap & \multicolumn{9}{c|}{ $n_{g,j}(13)$: driving terms of length $j$ in $\pgap(\primeprod{13})$ }\\ [1 ex]
$g$ & $j=1$ & $2$ & $3$ & $4$ & $5$ & $6$ & $7$ & $8$ & $9$ \\ \hline 
 $\lil 2, \; 4$ & $\lil 1485$ & & & & & & & & \\
 $\lil 6$ & $\lil 1690$ & $\lil 1280$ & & & & & & & \\
 $\lil 8$ & $\lil 394$ & $\lil 902$ & $\lil 189$ & & & & & & \\
 $\lil 10$ & $\lil 438$ & $\lil 1164$ & $\lil 378$ & & & & & & \\
 $\lil 12$ & $\lil 188$ & $\lil 1276$ & $\lil 1314$ & $\lil 192$ & & & & & \\
 $\lil 14$ & $\lil 58$ & $\lil 536$ & $\lil 900$ & $\lil 288$ & & & & & \\
 $\lil 16$ & $\lil 12$ & $\lil 252$ & $\lil 750$ & $\lil 436$ & $\lil 35$ & & & &  \\
 $\lil 18$ & $\lil 8$ & $\lil 256$ & $\lil 1224$ & $\lil 1272$ & $\lil 210$ & & & & \\
 $\lil 20$ & $\lil 0$ & $\lil 24$ & $\lil 348$ & $\lil 960$ & $\lil 600$ & $\lil 48$ & & & \\
 $\lil 22$ & $\lil 2$ & $\lil 48$ & $\lil 312$ & $\lil 784$ & $\lil 504$ & & & & \\
 $\lil 24$ & $\lil 0$ & $\lil 20$ & $\lil 258$ & $\lil 928$ & $\lil 1260$ & $\lil 504$ & & & \\
 $\lil 26$ & $\lil 0$ & $\lil 2$ & $\lil 40$ & $\lil 322$ & $\lil 724$ & $\lil 448$ & $\lil 84$ & & \\
 $\lil 28$ & $\lil 0$ & $\lil 0$ & $\lil 36$ & $\lil 344$ & $\lil 794$ & $\lil 528$ & $\lil 80$ & & \\
 $\lil 30$ & $\lil 0$ & $\lil 0$ & $\lil 10$ & $\lil 194$ & $\lil 1066$ & $\lil 1784$ & $\lil 816$ & $\lil 90$ & \\
 $\lil 32$ & $\lil 0$ & $\lil 0$ & $\lil 0$ & $\lil 12$ & $\lil 200$ & $\lil 558$ & $\lil 523$ & $\lil 172$ & $\lil 20$ \\ \hline
 \end{tabular}
 \end{center}
\renewcommand\arraystretch{1}

For $g=6$ there are driving terms of length $j=2$, so we have a $2$-dimensional system. 
\begin{eqnarray*}
\left[ \begin{array}{c} \Rat{1}{6} \\ \Rat{2}{6} \end{array} \right]_{p_k}
& = & \left[ \begin{array}{cc}
  \frac{p_k-2}{p_k-2} & \frac{1}{p_k-2} \\ & \\ 0 & \frac{p_k-3}{p_k-2} \end{array} \right]
  \cdot  \left[ \begin{array}{c} \Rat{1}{6} \\ \Rat{2}{6} \end{array} \right] _{p_{k-1}}  \nonumber \\
 & = & \left[ \begin{array}{cc}
 1 & b_1 \\ 0 & a_2 \end{array} \right]
  \cdot  \left[ \begin{array}{c} \Rat{1}{6} \\ \Rat{2}{6} \end{array} \right]_{p_{k-1}} 
\end{eqnarray*} 

We have the system matrix
$$ M_2 = \left[ \begin{array}{cc}
 1 & b_1 \\ 0 & a_2 \end{array} \right] $$ 
with $b_1 = b_1(p)  = \frac{1}{p-2}$ and $a_2 = a_2(p) = \frac{p-3}{p-2}$.  We will often suppress
the explicit dependence of $a_i$ and $b_i$ on the prime $p$, but a consequence is that multiplication
among these parameters does not commute.  

Formulated in this way, we can use common methods
of analysis for dynamic systems, except that the values of the matrix entries depend on the
progression of primes.
Again we caution that we have qualified the exponential notation, to mean the product of
a parameter over the appropriate sequence of prime numbers. Let 
\begin{eqnarray*}
\left[ \begin{array}{c} \Rat{1}{6} \\ \Rat{2}{6} \end{array} \right]_{p_k}
& = & \left. M_2 \right|_{p_k}
  \cdot  \left[ \begin{array}{c} \Rat{1}{6} \\ \Rat{2}{6} \end{array} \right]_{p_{k-1}}  \\
 & = & M_2^k
  \cdot  \left[ \begin{array}{c} \Rat{1}{6} \\ \Rat{2}{6} \end{array} \right]_{p_0} 
\end{eqnarray*} 
To understand the relative occurrence of $6$'s to $2$'s in the large, we examine the
matrices $M_2^k$.
$$ M_2^k = \left[ \begin{array}{cc}
 1 & \beta_{12}^{(k)} \\ 0 & a_2^k \end{array} \right]$$
 with initial values $\beta_{12} = b_1(17) = \frac{1}{15}$, $a_2 = \frac{14}{15}$, and powers
\begin{eqnarray*}
 \beta_{12}^{(k)}  &=&  1 \cdot \beta_{12}^{(k-1)} + \frac{1}{p_k - 2} \cdot a_2^{k-1} \\
{\rm and} \makebox[0.2 in]{}
 a_2^k  &=&  \frac{p_k-3}{p_k-2} \gap a_2^{k-1}  =  \prod_{q=p_1}^{p_k} \frac{q-3}{q-2}.
\end{eqnarray*}
 
 The limit of the ratios $\Rat{j}{6}$ is determined by the limit of 
 products of the system matrix
\begin{equation*}
M_2^\infty  =  \left[ \begin{array}{cc}
 1 & \beta_{12}^{(\infty)} \\ 0 & a_2^\infty \end{array} \right]  \; \;
    = \; \left[ \begin{array}{cc}
 1 & \lim_{k \rightarrow \infty} \beta_{12}^{(k)} \\
  0 & \lim_{k \rightarrow \infty} \prod_{q=17}^{p_k} \frac{q-3}{q-2}\end{array} \right].
\end{equation*}
  
For $g=8$ and $g=10$, there are driving terms up to length $3$, so we have a $3$-dimensional system. 
The system matrix is
$$ M_3 = \left[ \begin{array}{ccc}
 1 & b_1 & 0 \\ 0 & a_2 & b_2 \\ 0 & 0 & a_3 \end{array} \right] $$ 
with $b_1$ and $a_2$ as before in $M_2$, $b_2 = b_2(p)  = \frac{2}{p-2}$ 
and $a_3 = a_3(p) = \frac{p-4}{p-2}$.

Powers of $M_3$ will be upper triangular
$$ M_3^k = \left[ \begin{array}{ccc}
1 & \beta_{12}^{(k)} & \beta_{13}^{(k)} \\
0 & a_2^k & \beta_{23}^{(k)} \\
0 & 0 & a_3^k \end{array} \right] = \left. M_3 \right|_{p_k} \cdot M_3^{k-1},
$$
with the following recursive definitions:
\begin{eqnarray}
a_2^k & = & \prod_{q=17}^{p_k} \frac{q-3}{q-2} \label{EqA2} \\
a_3^k & = & \prod_{q=17}^{p_k} \frac{q-4}{q-2} \label{EqA3} \\
\nonumber
\beta_{12}^{(k)} & = & 1\cdot \beta_{12}^{(k-1)} + \frac{1}{p_k-2}\cdot a_2^{k-1} \\ 
\nonumber
\beta_{23}^{(k)} & = & \frac{p_k-3}{p_k-2} \cdot \beta_{23}^{(k-1)} + \frac{2}{p_k-2}\cdot a_3^{k-1} \\ \nonumber
\beta_{13}^{(k)} & = & 1 \cdot \beta_{13}^{(k-1)} + \frac{1}{p_k-2} \cdot \beta_{23}^{(k-1)}
\end{eqnarray}

\begin{figure}[b] 
\centering
\includegraphics[width=5in]{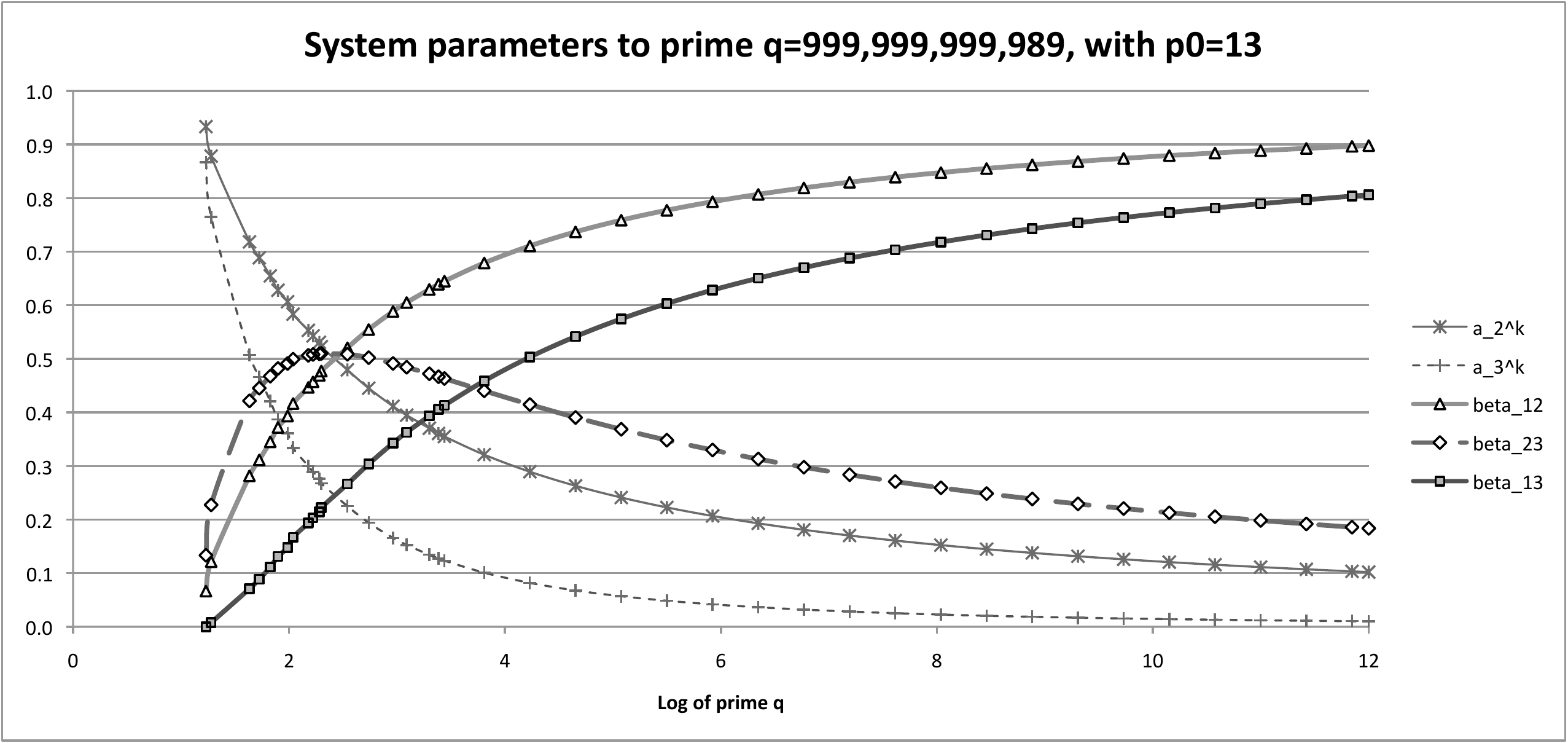}
\caption{\label{ParamFig} This figure illustrates the values of the system parameters
for $M_3^k$ as the
value of $p_k$ runs from $17$ to $999,999,999,989$.
With the parameters $\beta_{12}^{(k)}$ and $\beta_{13}^{(k)}$, we can calculate
the ratios $\Rat{j}{g}$ for the gaps  $6, \; 8, \; 10$ up through 
$\pgap(\primeprod{999,999,999,989})$.   }
\end{figure}

Since we will later be comparing these values to $\Rat{j}{30}$, we calculate initial 
conditions using $p_0=13$.  We can then use calculations of the system parameters 
in $M_3^k$ to obtain the ratios $\Rat{j}{8}$ and $\Rat{j}{10}$ for large primes.
With $p_0=13$, we have calculated the system parameters through
$p_k = \hat{p} = 999,999,999,989$.  See Figure~\ref{ParamFig}.  
For this value of $p_k$, we calculate the following values.  

\noindent
\begin{tabular}{rrrr} \hline
\multicolumn{4}{l}{For $p_0=13$ }  \\
 $\underline{g}$ & $\underline{\Rat{1}{g}(13)}$ & 
 $\underline{\Rat{2}{g}(13)}$ & $\underline{\Rat{3}{g}(13)}$ \\ 
$6$ & $1.13804714$ & $0.86195286$ & $0$ \\
$8$ & $0.26531987$ & $0.60740741$ & $0.12727273$ \\
 $10$ & $0.29494949$ & $0.78383838$ & $0.25454545$ \\ \hline
\multicolumn{4}{l}{For $p_k=\hat{p}=999,999,999,989$}  \\
  & $a_1^k=1$ & $\beta_{12}^k=0.89793248$ &  $\beta_{13}^k=0.80606493$  \\ 
  & & & \\
 & $\mathbf{\Rat{1}{6}(\hat{p}) = 1.91202}$  &
 $\mathbf{\Rat{1}{8}(\hat{p}) = 0.91332}$ &
  $\mathbf{\Rat{1}{10}(\hat{p}) = 1.20396}$  \\ \hline
\end{tabular} 

This data tells us that in $\pgap(\primeprod{999,999,999,989})$, which covers the interval 
$\hat{p}=999,999,999,989$ to 
$\primeprod{\hat{p}} \approx 10^{434294060804}$, the ratio of gaps $g=6$ to gaps $g=2$ is
$\Rat{1}{6}(\hat{p})= 1.91202$.  The number of gaps $g=10$ has surpassed the gaps $g=2$
with a ratio of $\Rat{1}{10}(\hat{p})=1.20396$, but the gaps $g=8$ still lag the number of 
gaps $g=2$ with a ratio $\Rat{1}{8}(\hat{p})=0.91332$.

\section{General system}
The general form of this dynamic system, for gaps or constellations with driving terms 
of length $j \le J$ is
\begin{eqnarray*}
\left[ \begin{array}{c}
\Rat{1}{g} \\ \vdots \\ \Rat{J}{g} \end{array} \right]_{p_k}
& = & \left[ \begin{array}{cccccc}
1 & b_1 & 0 & \cdots & & 0 \\
0  & a_2 & b_2 & \ddots &  & 0 \\
 & 0 & a_3 & b_3 & \ddots & 0 \\
  \vdots & & \ddots & \ddots & \ddots & 0 \\
  0 & & \cdots & & a_{J-1} & b_{J-1} \\
  0 & & \cdots & & 0 & a_J \end{array} \right]_{p_k} \cdot
  \left[ \begin{array}{c}
\Rat{1}{g} \\ \vdots \\ \Rat{J}{g} \end{array} \right]_{p_{k-1}} \\
 & = & \left. M_J  \right|_{p_k} \cdot
  \left[ \begin{array}{c}
\Rat{1}{g} \\ \vdots \\ \Rat{J}{g} \end{array} \right]_{p_{k-1}} 
\; =   M_J^k  \cdot
  \left[ \begin{array}{c}
\Rat{1}{g} \\ \vdots \\ \Rat{J}{g} \end{array} \right]_{p_0} 
\end{eqnarray*}

Each $\Rat{j}{g}(p_k)$ is the ratio of the number of driving terms of length $j$ for the gap $g$,
to the number of gaps $2$ in the cycle of gaps $\pgap(\primeprod{p_k})$.  In particular, 
$\Rat{1}{g}(p_k)$ is the ratio of the number of gaps $g$ to gaps $2$ at this stage of the recursion.

$M_J$ is a banded matrix that depends on the iteration $p_k$ but {\em not} on the gap $g$.  
\begin{eqnarray}\label{Eqab}
b_j &  = & \frac{j}{p-2} \\ \nonumber
a_j & =  & \frac{p-j-1}{p-2}
\end{eqnarray}
While $M_J$ is banded, $M_J^k$ becomes upper triangular. 
\begin{eqnarray*}
M_J^k & = & \left[ 
\begin{array}{ccccc}
1 & \beta_{12}^{(k)} & \beta_{13}^{(k)} & \cdots & \beta_{1J}^{(k)} \\
0 & a_2^k & \beta_{23}^{(k)} & \cdots & \beta_{2J}^{(k)} \\
\vdots & & \ddots & \ddots & \vdots \\
0 & & & a_{J-1}^k & \beta_{J-1,J}^{(k)} \\
0 & \cdots & & 0 & a_J^k \end{array} \right]
\end{eqnarray*}
with
\begin{equation*}
\beta_{ij}^{(k)}  = \left\{
\begin{array}{lcl}
a_i \cdot \beta_{ij}^{(k-1)} + b_i \cdot a_j^{k-1} & {\rm if} & j = i+1 \\
& & \\
a_i \cdot \beta_{ij}^{(k-1)} + b_i \cdot \beta_{i+1,j}^{(k-1)} & {\rm if} & j > i+1 \end{array} \right.
\end{equation*}
Note that the multiplication on the right-hand side does not commute, since the value of each
factor depends on the respective value of the prime $p$ as indicated by its position in the
product.

$M_J^k$ applies to all constellations whose driving terms have length $j \le J$; and we 
continue to use the exponential notation to denote the product over the sequence of primes from
$p_1$ to $p_k$: e.g.
$$ M_J^k =  \left. M_J\right|_{p_k} \cdot \left. M_J\right|_{p_{k-1}} \cdots \left. M_J\right|_{p_1}.
$$ 
With $M_J^k$ we can calculate the ratios $\Rat{j}{g}(p_k)$
for the complete system of driving terms, relative to the
population of the gap $2$, for the cycle of gaps $\pgap(\primeprod{p_k})$ (here, $p_k$ is the $k^{\rm th}$
prime after $p_0$).  With $J=3$ we calculated above the ratios for $g=6, \; 8, \; 10.$
For $g=12$ we need $J=4$, and for $g=30$, we need $J=8$.

Fortunately, we can completely describe the eigenstructure for $\left. M_J\right|_{p}$,
and even better -- {\em the eigenvectors for $M_J$ do not depend on the prime $p$}.  This means
that we can use the eigenstructure to describe the behavior of this iterative system as 
$k \fto \infty$.

\subsection{Eigenstructure of $M_J$}
We list the eigenvalues, the left eigenvectors and the right eigenvectors for $M_J$, writing these
in the product form
$$ M_J  =  R \cdot \Lambda \cdot L $$
with $LR = I$.

With $a_j$ and $b_j$ as defined in Equation \ref{Eqab}, for $J=4$ we have
\begin{eqnarray*}
M_4 & = & \left[ \begin{array}{cccc}
1 & b_1 & 0 & 0 \\
0 & a_2 & b_2 & 0 \\
0 & 0 & a_3 & b_3 \\
0 & 0 & 0 & a_4 \end{array}\right] \\
& = & \left[ \begin{array}{rrrr}
1 & -1 & 1 & -1 \\
0 & 1 & -2 & 3 \\
0 & 0 & 1 & -3 \\
0 & 0 & 0 & 1 \end{array}\right] 
\cdot \left[ \begin{array}{cccc}
1 & 0 & 0 & 0 \\
0 & a_2 & 0 & 0 \\
0 & 0 & a_3 & 0 \\
0 & 0 & 0 & a_4 \end{array}\right]
\cdot \left[ \begin{array}{cccc}
1 & 1 & 1 & 1 \\
0 & 1 & 2 & 3 \\
0 & 0 & 1 & 3 \\
0 & 0 & 0 & 1 \end{array}\right]
\end{eqnarray*}

Note that while the eigenvalues of $M_4$ depend on the prime $p$ (through the $a_j$), 
the eigenvectors do not.
Thus the matrix $M_4^k$ can be written
\begin{eqnarray*}
M_4^k & = &  R \Lambda^k L \\
& = & \left[ \begin{array}{rrrr}
1 & -1 & 1 & -1 \\
0 & 1 & -2 & 3 \\
0 & 0 & 1 & -3 \\
0 & 0 & 0 & 1 \end{array}\right] 
\cdot \left[ \begin{array}{cccc}
1 & 0 & 0 & 0 \\
0 & a_2^k & 0 & 0 \\
0 & 0 & a_3^k & 0 \\
0 & 0 & 0 & a_4^k \end{array}\right]
\cdot \left[ \begin{array}{cccc}
1 & 1 & 1 & 1 \\
0 & 1 & 2 & 3 \\
0 & 0 & 1 & 3 \\
0 & 0 & 0 & 1 \end{array}\right]
\end{eqnarray*}

With $J=4$ we can calculate the ratio of the gap $g=12$ to the gap $g=2$ in the cycle of gaps.
For initial conditions at $p_0 = 13$, we have
\begin{equation*}\begin{array}{lll}
\N{13}{2} = 1485 & \N{13}{12} = 188 & \Rat{1}{12}(13) = 188 / 1485 \\ %0.126599327 \\
 & n_{12,2}(13) = 1276 & \Rat{2}{12}(13) = 1276 / 1485 \\ %0.859259259 \\
 & n_{12,3}(13) = 1314 & \Rat{3}{12}(13) = 1314 / 1485 \\ %0.884848485 \\
 & n_{12,4}(13) = 192 & \Rat{4}{12}(13) = 192 / 1485 \\ %0.129292929
\end{array}
\end{equation*}

To determine the ratios after $k$ iterations of the recursion, we apply $M_4^k$.
\begin{eqnarray*}
M_4^k \cdot \left. \bar{w} \right|_{13} & = &  \left[ \begin{array}{rrrr}
1 & -1 & 1 & -1 \\
0 & 1 & -2 & 3 \\
0 & 0 & 1 & -3 \\
0 & 0 & 0 & 1 \end{array}\right] 
\cdot \left[ \begin{array}{cccc}
1 & 0 & 0 & 0 \\
0 & a_2^k & 0 & 0 \\
0 & 0 & a_3^k & 0 \\
0 & 0 & 0 & a_4^k \end{array}\right]
\cdot \left[ \begin{array}{cccc}
1 & 1 & 1 & 1 \\
0 & 1 & 2 & 3 \\
0 & 0 & 1 & 3 \\
0 & 0 & 0 & 1 \end{array}\right] \cdot  
\left[ {\small \begin{array}{l} % 0.126599327 \\ 0.859259259 \\ 0.884848485 \\  0.129292929
188 / 1485 \\ 1276 / 1485 \\ 1314 / 1485 \\ 192 / 1485
\end{array}} \right] \\ 
& = & \left[ \begin{array}{rrrr}
1 & -1 & 1 & -1 \\
0 & 1 & -2 & 3 \\
0 & 0 & 1 & -3 \\
0 & 0 & 0 & 1 \end{array}\right] 
\cdot  \left[ \begin{array}{l}
2 \\ {\small 4480 / 1485} \;  a_2^k \\ {\small 1890/ 1485} \; a_3^k \\ 
{\small 192 / 1485}  \; a_4^k \end{array}\right] 
\end{eqnarray*}

Focusing just on the ratio $\Rat{1}{12}$ of the occurrences of gap $g=12$ to $g=2$, we 
see that 
$$ \Rat{1}{12}(p_k) = 2 - \frac{4480}{1485} \;  a_2^k + \frac{1890}{1485} \; a_3^k - \frac{192}{1485} \; a_4^k 
$$
which converges to $\Rat{1}{12}(p_\infty) = 2$ as rapidly as $a_2^k \fto 0$.  In Figure~\ref{ParamFig}
we observe that $a_2^k$ still has a value around $0.1$ for $p_k \sim 10^{12}$.

For the general system $M_J$, the upper triangular entries of $R$ and $L$ are binomial coefficients,
with those in $R$ of alternating sign; and the eigenvalues are the $a_j$.
\begin{eqnarray*}
R_{ij} & = & \left\{ \begin{array}{cl}
(-1)^{i+j}\Bi{j-1}{i-1} & {\rm if} \; i \le j \\ & \\
0 & {\rm if} \; i > j \end{array} \right. \\
& & \\
\Lambda & = & {\rm diag}(1, a_2, \ldots, a_J) \\
 & & \\
L_{ij}  & = & \left\{ \begin{array}{ll}
\Bi{j-1}{i-1} & {\rm if} \; i \le j \\ & \\
0 & {\rm if} \; i > j \end{array} \right. 
\end{eqnarray*}

For any vector $\bar{w}$, multiplication by the left eigenvectors 
(the rows of $L$) yields the 
coefficients for expressing this vector of initial conditions over the basis given by the 
right eigenvectors (the columns of $R$):
$$ \bar{w} =  (L_{1 \cdot} \bar{w}) R_{\cdot 1} + \cdots + (L_{J \cdot} \bar{w}) R_{\cdot J}
$$

\begin{lemma}\label{AsymLemma}
For any gap $g$ with initial ratios $\bar{w}_0$, the ratio of occurrences of this gap $g$
to occurrences of the gap $2$ in $\pgap(\primeprod{p})$ as $p \fto \infty$ converges to
the sum of the initial ratios across the gap and all its driving terms:
$$
\Rat{1}{g}(\infty) = L_{1 \cdot} \bar{w}_0 = \sum_j \left. \Rat{j}{g} \right|_{p_0}.
$$
\end{lemma}

\begin{proof}
Let $g$ have driving terms up to length $J$.  Then the ratios $\left. \bar{w} \right|_{p}$ 
are given by  the iterative linear system
$$
\left. \bar{w} \right|_{p_k} = M_J^k \cdot  \bar{w}_0.
$$
From the eigenstructure of $M_J$, we have
\begin{eqnarray}
\bar{w}_0 &=& (L_1 \bar{w}_0)R_1 + (L_2 \bar{w}_0)R_2 + \cdots + (L_J \bar{w}_0)R_J ,
\nonumber \\
{\rm and \; so} \gap \gap & & \nonumber \\ 
M_J^k \bar{w}_0 &=& 
  (L_1 \bar{w}_0)R_1 + a_2^k (L_2 \bar{w}_0)R_2  + \cdots + a_J^k (L_J \bar{w}_0)R_J. 
\end{eqnarray}
We note that $L_{1 \cdot} = [1 \cdots 1]$, $\lambda_1 = 1$, and $R_{\cdot 1} = e_1$;
that the other eigenvalues $a_j^k \fto 0$ with $a_j^k > a_{j+1}^k$.   Thus as $k \fto \infty$
the terms on the righthand side decay to $0$ except for the first term, establishing the result.
\end{proof}

With Lemma~\ref{AsymLemma} and the initial values in $\pgap(\primeprod{13})$ 
tabulated above, we can calculate the asymptotic ratios
of the occurrences of the gaps $g = 6, 8, \ldots, 32$ to the gap $g=2$, and we provide the
intermediate values at $\hat{p}=999,999,999,989$ to give a sense of the rate of convergence.

\renewcommand\arraystretch{1.2}
\begin{center}
\begin{tabular}{ll} \hline
\multicolumn{2}{|c|}{Values of $a_j^k$ at $\hat{p}=999,999,999,989$} \\ \hline
 &
$a_2^k = 0.102067517997789430000$ \\
 & $a_3^k = 0.0101999689756664110000$ \\
$ a_j^k = \prod_{q=17}^{\hat{p}} \frac{q-j-1}{q-2} $ & $a_4^k = 0.00099592269918294960000$ \\
 & $a_5^k = 0.000094770935314020220000$ \\
 & $a_6^k = 0.00000876214163461868090000 $ \\
 & $a_7^k = 0.000000784081204999455720000 $ \\
 & $a_8^k =  0.000000067575616112121770000$ \\
 & $a_9^k =  0.00000000557283548473588330000$ \\ \hline
\end{tabular}
\end{center}
\renewcommand\arraystretch{1}

From these values, we see the decay of the $a_j^k$ toward $0$, but $a_2^k$ and $a_3^k$ 
are still making significant contributions when $p_k \approx 10^{12}$.

\section{Observations and conclusions}
We recall that these ratios apply to the gaps 
in the cycle of gaps $\pgap(\primeprod{p})$.  
These ratios are representative of the gaps that will survive to become
gaps between prime numbers \cite{FBHgaps, HRest}, but they are not direct calculations
of the gaps among primes.

To calculate the ratio $\Rat{1}{g}(p_k)$, which gives the relative number of occurrences of the 
gap $g$ to the gap $2$ at the stage of Eratosthenes sieve for $p_k$, 
we only need the parameters 
$\beta_{1j}^{(k)}$ from the top row of $M_J^k$, and the initial values $\Rat{j}{g}(p_0)$.
\begin{equation*}
\Rat{1}{g}(p_k)  =   \Rat{1}{g}(p_0) + \sum_{j=2}^{J} \beta_{1j}^{(k)} \cdot \Rat{j}{g}(p_0).
\end{equation*}
Given the simple eigenstructure of $M_J$, we can compute the $\beta_{1j}^{(k)}$ 
from $M^k = R \Lambda^k L$.

Brent \cite{Brent} computed the Hardy and Littlewood estimates
\cite{HL} for the occurrences of gaps among primes for gaps $g=2,\; 4, \ldots, 80$,
in the range $10^6$ to $10^9$.  
In the table below, we compare the actual ratios of the
occurrences of the gaps from $4\ldots 32$ to the occurrences of the gap $2$; to the ratios
in the predictions as computed by Brent; to the ratios of occurrences in the cycle of gaps
$\pgap(\primeprod{45053})$ -- we chose this prime as a representative whose square is
approximately $2\times 10^9$; to the ratios in the cycle of gaps for $\hat{p}=999,999,999,989$;
and to the asymptotic value.

\begin{center}
\begin{tabular}{|rrrr||r|rl|} \hline
\multicolumn{4}{|c|}{Counts and ests over $[10^6,10^9]$} & 
\multicolumn{3}{c|}{Ratios in $\pgap(\primeprod{p})$} \\ \hline
gap & actual & actual & Brent-HL & $w_{g,1}(45053)$ & $w_{g,1}(\hat{p})$ & $w_{g,1}(\infty)$ \\
 & count & ratio-to-2 & ratios & & & \\ \hline
2 & 3416337 & & & & & \\
4 & 3416536 & 1.000058 & 1.000000 & 1.000000 & 1.000000 & 1 \\
6 & 6076242 & 1.778584 & 1.778548 & 1.773251& 1.912023 & 2 \\
8 & 2689540 & 0.787258 & 0.786805 & 0.781874 & 0.913321 & 1 \\
10 & 3477688 & 1.017958 & 1.017669 & 1.010457 & 1.203964  & 1.3333 \\
12 & 4460952 & 1.305770 & 1.305407 & 1.290409 & 1.704932 & 2 \\
14 & 2460332 & 0.720167 & 0.720315 & 0.710307 & 0.991980 & 1.2 \\
16 & 1843216 & 0.539530 & 0.539307 & 0.530094 & 0.795251 & 1 \\
18 & 3346123 & 0.979448 & 0.979564 & 0.959984 & 1.536000 & 2 \\
20 & 1821641 & 0.533215 & 0.533624 & 0.519616 & 0.952118 & 1.3333 \\
22 & 1567507 & 0.458827 & 0.458646 & 0.447082 & 0.801923 & 1.1111 \\
24 & 2364792 & 0.692201 & 0.691456 & 0.670242 & 1.352488 & 2 \\
26 & 1118410 & 0.327371 & 0.327304 & 0.315738 & 0.701375 & 1.0909 \\
28 & 1218009 & 0.356525 & 0.356576 & 0.343838 & 0.769263  & 1.2 \\
30 & 2176077 & 0.636962 & 0.636843 & 0.609471& 1.580455 & 2.6667 \\
32 & 683346 & 0.200023 & 0.199842 & 0.190052  & 0.555727 & 1 \\ \hline
\end{tabular}
\end{center}

The values $w_{g,1}$ are the actual ratios between the numbers of these gaps at the corresponding
stage of Eratosthenes sieve.  So these ratios, when computed exactly, represent the exact proportions
of the relative occurrences among these gaps.

If there are significant deviations from these ratios among gaps in the cycle compared to the ratios of
those that survive to be gaps among primes over this range, what can we understand 
about the mechanism that would selectively close gaps of certain values?  

This column $\Rat{1}{g}(\hat{p})$ provides the ratios in $\pgap(\primeprod{999,999,999,989})$, 
which covers the interval $\hat{p}=999,999,999,989$ to 
$\primeprod{\hat{p}} \approx 10^{434294060804}$.
As the recursion continues, many closures will occur within this range.  
The final column $\Rat{1}{g}(\infty)$ provides the asymptotic ratios of the occurrences of the
indicated gap to the occurrences of the gap $2$.  

To understand the convergence to $\Rat{1}{g}(\infty)$, from the eigenstructure of $M_J^k$
we can approximate $\Rat{1}{g}(p_k)$ by truncating:
$$
\Rat{1}{g}(p_k) \approx 1\cdot \sum_{j=1}^{J} \Rat{j}{g}(p_0)
 - a_2^k \cdot \sum_{j=2}^{J} (j-1)\Rat{j}{g}(p_0) 
 + a_3^k \cdot \sum_{j=3}^{J} \binom{j-1}{2}\Rat{j}{g}(p_0) \ldots
 $$
Note that for $\pgap(\primeprod{999,999,999,989})$
the value of $a_2^k \approx 0.1$, so the convergence to $0$ is very gradual.

This work supports the conjecture that $30$ eventually is a more common gap among primes than $6$.
In the table above, we see that asymptotically there are in the cycles of gaps for Eratosthenes sieve
twice as many $6$'s as $2$'s and $2\frac{2}{3}$ times as many $30$'s as $2$'s.  
However, even at the prime
$\hat{p} \approx 10^{12}$, these ratios are $\Rat{1}{6}(\hat{p})= 1.91202$ and
 $\Rat{1}{30}(\hat{p})=1.580455$.
 Truncating $\Rat{1}{g}(p_k)$ as suggested and using the initial conditions for $g=6$ and $g=30$ 
 in $\pgap(\primeprod{13})$, we see that $30$'s will outnumber $6$'s in Eratosthenes sieve 
 when $a_2^k < 0.07$.
 
 The asymptotic ratios appear to follow the formula:
\begin{equation*}
 \Rat{1}{g}(\infty) = \prod_{q | g, \gap q>2} \frac{q-1}{q-2}.
 \end{equation*}
 It would be interesting to see whether this formula holds up for larger gaps, since it provides supporting
 evidence for Conjecture B in \cite{HL}; these ratios among gaps hold asymptotically in Eratosthenes sieve.  
 From Lemma~\ref{AsymLemma} this means that for a given set of prime factors (no
 matter what the powers on these factors), any gap with this same set of prime factors has the
 same total number of driving terms in any stage of Eratosthenes sieve that satisfies the conditions
 of Theorem~\ref{CountThm}.

One more observation about primorial gaps and their driving terms.
Since the length of $\pgap(\primeprod{5})$ is $8$ with sum $30$, in all subsequent cycles of gaps
the sum of every constellation of length $8$ will be at least $30$.  Since $n_{30,8}(\primeprod{13})=90$, there are $90$ complete copies of $\pgap(\primeprod{5})$ in $\pgap(\primeprod{13})$.  Complete copies are only preserved for $\pgap(\primeprod{5})$, $\pgap(\primeprod{3})$, and $\pgap(\primeprod{2})$.
These are preserved since the elementwise products in step R3 of the recursion are large enough to
pass completely over one of the copies concatenated in step R2.  Starting with $\pgap(\primeprod{7})$, 
the primorial $\primeprod{7}$ is larger than any of the elementwise products, and no complete copies of these longer cycles are preserved in their entirety.

%=======================================

\bibliographystyle{amsplain}

% \bibliography{primes}

\begin{thebibliography}{10}

\bibitem{Brent}
R.P.~Brent, \emph{The distribution of small gaps between successive prime
  numbers}, Math. Comp. \textbf{28} (1974), 315--324.

\bibitem{HL}
G.H. Hardy and J.E. Littlewood, \emph{Some problems in 'partitio numerorum'
  iii: On the expression of a number as a sum of primes}, G.H. Hardy Collected
  Papers, vol.~1, Clarendon Press, 1966, pp.~561--630.

\bibitem{FBHgaps}
F.B. Holt, \emph{Expected gaps between prime numbers}, arXiv 0706.08889v1,
6 June 2007.

\bibitem{HRest}
F.B. Holt and H. Rudd, \emph{Estimating constellations among primes - I. uniformity},
arXiv 1312.2165, 8 Dec 2013.

\bibitem{MV}
H.L. Montgomery and R.C. Vaughan, \emph{On the distribution of reduced
  residues}, Annals of Math., 2nd series \textbf{123} (1986), no.~2, 311--333.

\bibitem{Riesel}
H.~Riesel, \emph{Prime numbers and computer methods for factorization}, 2 ed.,
  Birkhauser, 1994.


\end{thebibliography}
\providecommand{\bysame}{\leavevmode\hbox to3em{\hrulefill}\thinspace}
\providecommand{\MR}{\relax\ifhmode\unskip\space\fi MR }
\providecommand{\MRhref}[2]{%
  \href{http://www.ams.org/mathscinet-getitem?mr=#1}{#2}
}
\providecommand{\href}[2]{#2}

\end{document}